\documentclass[11pt,letterpaper,english]{smfart}
\usepackage[utf8]{inputenc}
\usepackage[english,french]{babel}
\usepackage{smfthm}
\usepackage{mathrsfs}
\usepackage{graphics}
\usepackage{hyperref} 
\usepackage{amssymb, amsmath, mathabx}
\usepackage{lmodern}
\usepackage[all]{xy}
\usepackage{multicol}
\usepackage{color}
\usepackage[shortlabels]{enumitem}
\usepackage{scalerel}
\usepackage{tikz-cd}
\makeindex
\DeclareUnicodeCharacter{00A0}{\relax}
\allowdisplaybreaks

\author{Antoine Ducros}
\address{Sorbonne Université, Université Paris-Diderot, CNRS, Institut de Mathématiques de Jussieu-Paris
Rive Gauche, Campus Pierre et Marie Curie, case 247, 4 place Jussieu, 75252 Paris cedex 5, France}
\email{antoine.ducros\at imj-prg.fr}
\urladdr{https://webusers.imj-prg.fr/$\sim$antoine.ducros/}

\title{Automatic meromorphy in non-archimedean geometry}
\begin{abstract}
In this article, we prove that if $U$ is a Zariski-open subset of a reduced non-archimedean analytic space $X$ and $f$ is an analytic function on $U$
whose zero-locus is equal to $Z\cap U$ for some Zariski-closed susbet $Z$ of $X$, then $f$ extends to a meromorphic function on $X$ (unique
if $U$ is dense). As a corollary, we prove that if $\mathscr X$ is a reduced scheme locally of finite type over an affinoid algebra,
every analytic function on $\mathscr X\an$ with algebraic zero-locus is algebraic. 

\end{abstract}

\newcommand{\eg}{e.\@g.\@}

\newcommand{\ie}{i.\@e.\@}

\newcommand{\resp}{resp.\@~}

\newcommand{\Prop}{Prop.\@~}

\newcommand{\Th}{Thm.\@~}

\NumberTheoremsAs{subsection}
\SwapTheoremNumbers

\newcommand{\an}{^{\mathrm{an}}}
\newcommand{\hr}[1]{\mathscr H(#1)}
\renewcommand{\d}{\mathrm d}
\newcommand{\abs}[1]{\mathopen|#1\mathclose|}

\newcommand{\A}{\mathbf A}
\newcommand{\C}{\mathbf C}

\newcommand{\R}{\mathbf R}
\newcommand{\Z}{\mathbf Z}

\renewcommand{\epsilon}{\varepsilon}
\renewcommand{\phi}{\varphi}

\renewcommand{\theequation}{\alph{equation}}
\SetEnumerateShortLabel{1}{\textnormal{(\arabic{enumi})}}
\SetEnumerateShortLabel{i}{\textnormal{(\roman{enumi})}}
\SetEnumerateShortLabel{a}{\textnormal{(\alph{enumi})}}
\SetEnumerateShortLabel{A}{\textnormal{(\Alph{enumi})}}
\SetEnumerateShortLabel{2}{\textnormal{(\arabic{enumii})}}
\SetEnumerateShortLabel{j}{\textnormal{(\roman{enumii})}}
\SetEnumerateShortLabel{b}{\textnormal{(\arabic{enumi}\alph{enumii})}}
\SetEnumerateShortLabel{c}{\textnormal{(\alph{enumii})}}
\SetEnumerateShortLabel{B}{\textnormal{(\Alph{enumii})}}

\begin{document}
\maketitle

\section{Introduction}

The purpose of this text is to prove the following theorem and its corollary; analytic
spaces here have to be understood in the sense of Berkovich.

\begin{theo}\label{main}
Let $k$ be a complete, non-archimedean field and let $X$ be a reduced $k$-analytic space. Let $U$
be a  Zariski-open subset of $X$. Let $f$ be an analytic function on $U$. 
The following are equivalent : 

\begin{enumerate}[i]
\item The function $f$ can be extended to a meromorphic function on $X$. 
\item The zero-locus of $f$ is of the form $Z\cap U$ for 
$Z$ a Zariski-closed subset of $X$ (\eg, $f$ is invertible on $U$). 
\end{enumerate}
\end{theo}

\begin{enonce}[remark]{Comments}
Implication (i)$\Rightarrow$(ii) is easy. So the actual content of Theorem \ref{main}
is 
(ii)$\Rightarrow$(i). Note that for an
analytic function $f$ on $U$ to fulfill condition (ii)
it suffices that the zero-locus $T$ of $f$ in $U$ be closed in $X$. Indeed, if this is the case, 
$T=T\cap U$ is Zariski-closed in $X$ as checked on the open cover of $X$ by $U$ and $X\setminus T$. 
\end{enonce}

\begin{coro}\label{cor}
Let $X$ be a reduced scheme locally of finite type over an affinoid 
algebra 
and let $f$ be an analytic function on $X\an$. The following are equivalent: 

\begin{enumerate}[i]
\item The function $f$ 
is algebraic; \ie, it belongs to (the image of) $\mathscr O_X(X)$.
\item The zero-locus of $f$ is algebraic, \ie, of the form
$Y\an$ for $Y$ a Zariski-closed subset of $X$
(\eg, $f$ is an invertible analytic function on $X\an$). 
\end{enumerate}
\end{coro}

\begin{enonce}[remark]{Comments}
Implication (i)$\Rightarrow$(ii) is obvious. So the actual content of Corollary \ref{cor}
is 
(ii)$\Rightarrow$(i). 
Note that for an analytic function $f$ on $X\an$ to fulfill condition (ii)
it suffices that $X$
be separated and that the zero-locus $T$ of $f$ in $X\an$ be
compact. Indeed, if this is the case, 
then by compactness $T$
is contained in $U\an$ for some
quasi-compact open subscheme $U$ of $X$. 
We can 
choose by Nagata's Theorem a proper
compactification $\overline U$ 
of $U$ over $A$. Then $T$ is a Zariski-closed subset of $\overline U\an$ as checked on the open
cover
of $\overline U\an$ by $U\an$
and $\overline U\an \setminus T$, and so $T$ is
algebraic by GAGA, see
for example \cite[Appendix A]{poineau2010}. 
\end{enonce}

\begin{rema}
In Theorem \ref{main}, the meromorphic extension of $f$ is not unique in general (think of the case $U=\emptyset$ !), 
but it is as soon as $U$ is dense in $X$. 
\end{rema}

\begin{rema}\label{false-over-C}
Theorem \ref{main}
exhibits a typical non-archimedean feature. Indeed, it fails definitely 
in the complex-analytic setting, as well as
Corollary \ref{cor}, as witnessed by the exponential function on $\mathbf A^{1,\mathrm{an}}_{\mathbf C}$. 
\end{rema}

\begin{rema}
The reducedness assumption is necessary for Theorem \ref{main}
and
Corollary \ref{cor}
to hold, at least when $k$ is not trivially valued. 
Indeed, assume that
$\abs{k^\times}\neq \{1\}$ and choose
$f=\sum a_it^i\in k[\![t]\!]$ with infinite radius
of convergence 
which is not a polynomial.  Then 
$1+\epsilon f$ is an invertible analytic function
on the $k$-analytic
space $\mathbf A^{1,\mathrm{an}}_{k[\epsilon]}$ with
$\epsilon$ a non-zero nilpotent element. And $f$ is not 
algebraic, nor does it admit any meromorphic extension 
to $\mathbf P^{1,\mathrm{an}}_{k[\epsilon]}$. 

If $k$ is trivially valued, Theorem \ref{main}
and Corollary \ref{cor}
also fail in general for non-reduced
spaces, as witnessed by the former example
over the field $K:=k(\!(s)\!)$ (equipped with any $s$-adic
absolute value), the point being that $K[\epsilon]$ is
a (non-strict) $k$-affinoid algebra. 
 But if $X$ is any
(possibly non-reduced)
$k$-scheme 
locally of finite type
over $k$ itself (rather than over a $k$-affinoid algebra), then
every analytic function on 
$X\an$ is algebraic -- one checks it on affine schemes by reducing to
the case of the $n$-dimensional affine space, for which this
is obvious. 
\end{rema}

\begin{rema}\label{question-piotr}
When the ring of integers
$k^\circ$ is a discrete valuation ring, Piotr Achinger has suggested 
the following possible analogue
of Theorem \ref{main} for a special formal scheme 
$\mathfrak X$ over $k^\circ$ 
(see \cite{berkovich1996}, \S1 for a definition) : let $f$ be an analytic function 
on $\mathfrak X_\eta$ whose zero-locus is of the form $\mathfrak Y_\eta$ for a Zariski-closed
formal subscheme $\mathfrak Y$ of $\mathfrak X$; then $f$ is bounded-- thus
in some sense meromorphic from the formal
viewpoint. This would be a beautiful and natural result, but we unfortunately do 
not know whether it holds. 
Adapting our strategy to this situation would require
one to establish formal analogues of two results
of
Bartenwerfer that play a key role in our proof (see Proposition \ref{bartenwerfer}). 
\end{rema}

\subsection*{Strategy  of the proof}
As noticed by one of the referees, the proof ultimately consists in establishing a slightly easier particular case of Theorem \ref{main}
and then in showing that meromorphy "can be detected curvewise"
and even "discwise". We shall more precisely prove the following results, from which 
Theorem \ref{main} will follow straightforwardly; we will then get Corollary \ref{cor} by combining
Theorem \ref{main} and GAGA results about meromorphic functions. 

\begin{lemm}[Particular case of Theorem \ref{main}]\label{lemma-disc}
Let $D$ be a (closed or open) one-dimensional disc and let $U\subset D$  be the complement of the origin. 
Let $f$ be an analytic function on $U$. The following are equivalent: 

\begin{enumerate}[i]
\item the function $f$ can be extended to a meromophic function on $D$; \ie,  $f$ has no essential singularity at the origin; 
\item the zero-locus of $f$ is of the form $Z\cap U$ for $Z$ a Zariski-closed subset of $D$; \ie, $f$ is invertible
on a punctured neighborhood of the origin. 
\end{enumerate}
\end{lemm}

\begin{theo}[Discwise detection of meromorphy]\label{theo-discwise}
Let $X$ be a reduced
$k$-analytic space, let $U$ be a Zariski-open subset of $X$ and let $f$ be an
analytic function on $U$. The following are equivalent: 

\begin{enumerate}[i]
\item the function $f$ can be extended to a meromorphic function on $X$ ; 
\item for every complete  extension $L$ of $k$, every one-dimensional
(closed or open) disc $D$ over $L$, and every $k$-morphism 
$\phi \colon D\to X$ such that $\phi^{-1}(U)=D\setminus \{0\}$, the function $\phi^*f$ on $D\setminus \{0\}$ admits 
an extension to a meromorphic function on $D$. 
\end{enumerate}
\end{theo}

\subsection{About the proof of Lemma \ref{lemma-disc}} 
The direct implication is obvious. 
For the converse implication 
one may assume by shrinking $D$
that $f$ is invertible on $U$. The Lemma then follows easily from the 
very classical fact that the power series $f$ admits a dominant monomial.
This is in some sense the core argument on which Theroem \ref{main} relies, 
and the only one that is specific to  the non-archimedean
world, and definitely prevents our method to be adapted over $\C$, see 
Remark \ref{false-over-C}.

\subsection{Quick overview of the proof of Theorem \ref{theo-discwise}}
The direct implication is straightforward. 
Let us assume (ii) and prove (i). 
By arguing componentwise 
one first reduces to the case where
$U$ is dense. The direct implication is then easy. 
For the converse one, one can argue G-locally and thus assume that
$X$ is affinoid. 
Up to performing a suitable radicial
ground field extension and modding out by nilpotents
we can moreover assume that the normalization of $X$
is geometrically normal, and that the reduced irreducible
components of $X\setminus U$ are genercially quasi-smooth.

Relying upon these two facts we then reduce, 
by using 
normalization
and the previously known
non-archimedean analogue
(due to Bartenwerfer) 
of the classical 
extension theorem for arbitrary  meromorphic functions
through Zariski-closed subsets 
of codimension $\geq 2$ on a normal space, 
to the case where $X$ is quasi-smooth and
where $U$ is the complement of a
quasi-smooth hypersurface $S$; and then, 
by looking at the local structure
of the embedding $S\hookrightarrow X$, to the 
the case where $X$ is a product $D\times_k Y$
with $D$ a one-dimensional closed disc
and $Y$ smooth and irreducible,
and where $S=Y\times\{0\}$.
We then pick a Zariski-generic point $y$
of $Y$ and denote by $\phi\colon D_{\hr y}\to Y$ the embedding of the fiber 
at $y$ of the projection map $X\to Y$. By assumption (ii) $\phi^*f$ 
has a meromorphic singularity at the origin of $D_{\hr y}$; this means that 
the coefficients of the power series defining $f$ vanish at $y$ for sufficiently
negative exponents. By genericity of $y$, the corresponding coefficients are actually zero
and $f$ is meromorphic.

\subsection{Structure of the paper}
Section \ref{reminders} is devoted to the 
material in non-Archimedean geometry that will be used in our proofs. 
The latter are presented in section \ref{proofs}; the reader may 
go directly to section \ref{proofs} and refer to section \ref{reminders}
when needed.

\section{Reminders on analytic geometry}\label{reminders}
 
\subsection{General references}
Our framework is that of Berkovich spaces. 
We refer the reader to the first chapters of \cite{berkovich1990}
and \cite{berkovich1993} for the basic definitions. \cite[section 3]{ducros2009}
and \cite[chapter 2]{ducros2018}
for the ``commutative
algebra properties"  (like being normal, reduced, etc.), \cite[section 4]{ducros2009}
for the Zariski topology, \cite[section 6]{ducros2009} and \cite[chapter 5]{ducros2018} 
for the notion of quasi-smoothness, \cite[section 5]{ducros2009}
for the normalization, and \cite[\S 2.6]{berkovich1993}
for the analytification of a scheme. 

Let us mention that contrary to Berkovich we use the notation $\mathscr O_X$ (rather
than $\mathscr O_{X_{\mathrm G}}$) for the structure sheaf
of $X$ \textit{for the G-topology}, so that $\mathscr O_X(V)$ makes sense for any
analytic domain $V$ of $X$. 

\subsection{Analytic functions on annuli}\label{functions-annulus}
In some sense, the core argument in the proof of Theorem \ref{main},
and the only step in which we show by kind of an explicit computation that some function 
is meromorphic, relies on the well-known description of invertible functions on 
(relative) annuli, which we have chosen to recall here for the reader's convenience. 

Let $R_1$ and $R_2$ be two positive real numbers with 
$R_1\leq R_2$. The annulus $\{R_1\leq \abs T\leq R_2\}\subset \A^{1,\mathrm{an}}_k$
is an affinoid space whose algebra of analytic functions is the set
$k\{T/R_2, R_1T^{-1}\}$ of power series
$\sum_{i\in \Z}a_i T^i$ 
with coefficients in $k$
such that $\abs {a_i}R_1^i\to 0$ when $i\to -\infty$ and $\abs {a_i}R_2^i\to 0$ when 
$i\to +\infty$, which can be rephrased by saying that $\abs{a_i}r^i\to 0$ for $\abs i\to +\infty$
for every $r\in [R_1,R_2]$ (the Banach norm of $k\{T/R_2, R_1T^{-1}\}$
maps $\sum a_i T^i$ to the maximum
of $\abs{a_i}r^i$ for $i\in \Z$ and $r\in [R_1,R_2]$). More generally 
if $X=\mathscr M(A)$ is an affinoid
algebra then $\{r_1\leq \abs T\leq R_2\}
\times_k X$ is affinoid and its algebra of
analytic functions is $k\{T/R_2,R_1T^{-1}\}\widehat \otimes_k A$ 
which is the set of power series $\sum a_i T^i$ with coefficients in $A$
such that $\|a_i\|_\infty r^i\to 0$
for $\abs i\to \infty$ for every $r\in [R_1,R_2]$. 

Now let $I$ be an arbitrary non-empty interval of $\R_+^\times$. By exhausting $I$
with compact intervals and using the G-sheaf property of analytic functions, we see that
the ring of analytic functions on $\{ \abs T\in I\}\subset \A^{1,\mathrm{an}}_k$ is
the set
of power series
$\sum_{i\in \Z}a_i T^i$ with $a_i\in k$ for all $i$ 
such that $\abs{a_i}r^i\to 0$ for $\abs i\to +\infty$
for every $r\in  I$; more generally, the ring of analytic functions on
$X\times_k \{ \abs T\in I\}$ is
the set
of power series
$\sum_{i\in \Z}a_i T^i$ with $a_i\in A$ for all $i$ 
such that $\|a_i\|_\infty^i\to 0$ for $\abs i\to +\infty$
for every $r\in  I$. 

\begin{lemm}\label{dominant-monomial}
Let $I$ be a non-empty interval of $\R_+^\times$ and let $f=\sum a_i T^i$
be an analytic function on $\{\abs T\in I\} \subset \A^{1,\mathrm{an}}_k$. 
The following are equivalent: 

\begin{enumerate}[i]
\item there exists $j$ such that $\abs{a_i}r^i>\abs{a_j}r^j$ for all $i\neq j$
and all $r\in I$ ; 
\item $f$ is invertible. 
\end{enumerate}
\end{lemm} 

\begin{proof}
If (i) holds then $f$ can be written $a_jT^j (1+u)$ with $\abs u<1$ everywhere on 
$\{\abs T\in I\}$, so $f$ is invertible (and $f^{-1}=a_j^{-1}T^{-j}\sum_\ell u^\ell$). 
Now assume that $f$ is invertible, and let us prove (i). We start by handling the particular
case where $I$ is a singleton $\{r\}$. In order to prove (i) we may enlarge the ground field 
and rescale $T$ and $f$, so we can assume that $r=1$ and then that 
$\|f\|=\max_i \abs{a_i}=1$. As $f$ is invertible, its image under the reduction map
$k\{T,T^{-1}\}\to \widetilde k[T,T^{-1}]$ is invertible as well, so is of the form 
$\alpha T^i$ for some $i\in \Z$ and $\alpha\in \widetilde k^\times$. Then $\abs{a_i}=1$
(and $\widetilde{a_i}=\alpha$) and $\abs{a_j}<1$ for every $j\neq i$, whence (i). 

Now let us deal with general $I$. By the above for every $r\in I$ there is 
some integer $i(r)$ such that 
$\abs{a_{i(r)}}r^i>\abs{a_j}r^j$ for every $j\neq i(r)$, and by connectedness of $I$ it suffices to prove that $r\mapsto 
i(r)$ is locally constant.
But this is a straightforward consequence of the fact that $\abs{a_i}s^i\to 0$ when $\abs i\to \infty$
for every $s\in I$. 
\end{proof}

The following lemma describes the local structure
of a pair $(S,X)$ where $X$ is a quasi-smooth affinoid domain and $S$ a quasi-smooth
closed analytic subspace of $X$. It is certainly well-known (such a description is for
instance carried out in the proof of \cite[Theorem 9.1]{berkovich1999}), but to our knowledge
it is not stated explicitly in the litterature, so we have chosen to write it down here. 

\begin{lemm}\label{structure-smooth}
Let $X$ be a quasi-smooth affinoid space over a non-archimedean field $k$, and les 
$S$ be quasi-smooth closed analytic subspace of $X$. Let $x\in S$ and let $d$ 
be the codimension of $S$ in $X$ at $d$. There exists an affinoid neighborhood $V$ of $x$ in $X$, a one-dimensional closed disc $D$, and an isomorphism
$V\simeq D^d\times_k(S\cap V)$ whose restriction to
$S\cap V$ is the composition $(S\cap V)\simeq\{0\}\times_k  (S\cap V)
\hookrightarrow  D^d\times_k(S\cap V)$. 
\end{lemm}

\begin{proof}
Let $I$ be the ideal defining 
$S$ and set $n=\dim_x X$. 
The $\hr x$-vector space $\Omega_{S/k}\otimes \hr x$ 
is of dimension $n-d$, and is the quotient of  the $n$-dimensional
space $\Omega_{X/k}\otimes \hr x$ by the subspace
generated by the $\d g\otimes 1$ for $g\in I$. Therefore the latter
subspace is $d$-dimensional 
and generated by $\d g_1\otimes 1,\ldots, \d g_d\otimes 1$ for some $g_1,\ldots,g_d$. 
 Let $S'$ be the Zariski closed subspace of $X$ 
defined by the ideal 
$(g_1,\ldots, g_d)$. It contains $x$
and is of dimension at least
$n-d$ at $x$ by the \textit{Hauptidealsatz}, 
and $\Omega_{S'/k}\otimes \hr x$ is of dimension $n-d$ by construction. Therefore $S'$ is quasi-smooth of dimension
$n-d$ at $x$, so there exists an affinoid neighborhood $V$ of $x$ such that $S'\cap V$ is quasi-smooth and irreducible
of dimension $n-d$. Then $S\cap V$
is a closed analytic subspace of $S'\cap V$
of dimension $n-d$ at $x$, which forces the equality 
$S\cap V=S'\cap V$. Otherwise said $S\cap V$ is defined 
as a closed analytic subspace of $V$ by the equations $g_1=0,\ldots, g_d=0$. 
Pick
analytic functions $g_{d+1}, \ldots, g_n$ on $V$ such that the $\d g_i\otimes 1$ generate 
$\Omega_{S/k}\otimes \hr x$. Then $\d g_1\otimes 1,\ldots, \d g_n\otimes 1$ generate $\Omega_{X/k}
\otimes \hr x$. The family $(g_1,\ldots, g_n)$ defines a map $p$ from $V$ to $\mathbf A^{n,\mathrm{an}}_k$, which 
by compactness takes value in $E^n$ for some suitable one-dimensional closed disc $E$. Since $S\cap V$ is described by the equations $g_1=0,\ldots, g_d=0$, 
one has $S\cap V=p^{-1}
(\{0\}\times E^{n-d})$.

The maps $p\colon V\to E^n$ and $p|_{S\cap V}\colon S
\cap V\to\{0\}\times E^{n-d}\simeq E^{n-d}$
are quasi-étale at $x$ by 
construction and in view of \cite[Lemma 5.4.5]{ducros2018}, so we can shrink $V$ around $x$ so that
both maps are quasi-étale. 
Let $q$ be the quasi-étale map 
$\mathrm{Id}\times p|_{S\cap V}\colon E^d\times (S\cap V)\to E^n$. 
By construction, $p^{-1}(\{0\}\times E^{n-d})$ and $q^{-1}(\{0\}\times E^{n-d})$ are isomorphic as
quasi-étale spaces over the closed analytic subspace $\{0\}\times E^{n-d}$ of $E^n$, since
they both can be identified with
\[\begin{tikzcd}S\cap V\ar[rr,"p|_{S\cap V}"']&&\{0\}\times E^{n-d}\end{tikzcd}.\]
It follows by \cite[Lemma 2.7]{ducros2023}
(which itself relies on the henselian property of a Berkovich space along a Zariski-closed subspace, 
see \cite[\Prop 4.3.4]{berkovich1993})
that there exists an analytic neighborhood $U$ of $\{0\}\times E^{n-d}$ in $E^n$ such that $p^{-1}(U)$ and $q^{-1}(U)$ are
$U$-isomorphic. Up to shrinking $U$ we can 
assume that it is of the form $D^d\times E^{n-d}$ where 
$D$ is a closed one-dimensional disc. 
We thus get an isomorphism between 
$\{v\in V, \abs{g_i(v)}\leq r, i=1,\ldots, d\}$ and $D^d\times_k(S\cap V)$ 
whose first component is $(g_1,\ldots, g_d)$
and such that the induced isomorphism between $(g_1,\ldots, g_d)^{-1}(0)=S\cap V$ and
$\{0\}\times_k(S\cap V)$ is the obvious one. 
\end{proof}

\subsection{Meromorphic functions}
One can define straightforwardly a G-sheaf of meromorphic functions
$\mathscr K_X$ on any 
analytic space $X$ 
\cite[2.23]{ducros2021a}; there is a natural embedding
$\mathscr O_X\hookrightarrow \mathscr K_X$ ; if $V=\mathscr M(A)$
is an affinoid domain 
of $X$ then $\mathscr K_X(V)$ is the total ring of fractions of $A$ 
\cite[2.23.4]{ducros2021a}. As in the scheme-theoretic case, every meromorphic
funtion $f$ has a \textit{sheaf of denominators}; this is 
the coherent
sheaf of ideals
whose sections on an analytic domain $V$ consist of those $g\in \mathscr O_X(V)$
such that $gf\in \mathscr O_X(V)$.

\subsection{}
Let $A$ be an affinoid algebra, let
$X$ be an $A$-scheme of finite type and let $\mathscr K_X$ be 
the sheaf of
meromorphic (or rational) functions on $X$. 
Let $\mathscr S$, \resp $\mathscr T$, 
be the subsheaf of $\mathscr O_X$, \resp $\mathscr O_{X\an}$, consisting of those functions
whose germ at every point is a regular element of the corresponding
local ring. The structure map 
$\pi \colon X\an\to X$ is a faithfully flat map of locally ringed spaces. This implies that 
the natural arrow $\mathscr O_X\to \pi_*\mathscr O_{X\an}$ is injective, maps $\mathscr S$
into $\pi_*\mathscr T$, and induces an injective map of \textit{presheaves} $\mathscr S^{-1}\mathscr O_X
\to
\pi_*\mathscr T^{-1}\mathscr O_{X\an}$. The presheaf $\mathscr T^{-1}\mathscr O_{X\an}$ is separated, 
thus it embedds ito its sheafification $\mathscr K_{X\an}$, whence an injective map 
from
$\mathscr S^{-1}\mathscr O_X\hookrightarrow $ to $\pi_*\mathscr K_{X\an}$ and
eventually by sheafifyfing once again
an injective map $\mathscr K_X\hookrightarrow \pi_*\mathscr K_{X\an}$. 

\begin{lemm}\label{deno}
Let $f\in \mathscr K_X(X)$ and let $\mathscr I\subset \mathscr O_X$
be its sheaf of denominators. Then $\mathscr I\an$ is the sheaf of denominators of $f$
viewed
as a meromorphic function on $X\an$. 
\end{lemm}

\begin{proof}
If $\mathscr J$ denotes the sheaf of denominators of $f$ viewed as 
a meromorphic function on $X\an$, it is clear that  $\mathscr I\an
\subset \mathscr J$. To show that this inclusion is actually an isomorphism, one may
argue locally on $X$, and thus assume that $X$ is affine. One can then write
$f=g/h$ with $g$ and $h$ in $\mathscr O_X(X)$ and $h$ regular. 
Then $\mathscr I$ is 
the kernel of the map
$\mathscr O_X\to \mathscr O_X/(h)$ induced by multiplication by $g$,
and $\mathscr J$
is te kernel of the map
$\mathscr O_{X\an}\to \mathscr O_{X\an}/(h)$ induced by multiplication by $g$,
hence $\mathscr J=\mathscr I\an$.
\end{proof}

\begin{lemm}\label{mero-holo}
One has $\mathscr O_X(X)=
\mathscr K_X(X)\cap \mathscr O_{X\an}(X\an)
$.
\end{lemm}

\begin{proof}
The direct inclusion is obvious. For the converse one, let $f$ be an element of 
$\mathscr K_X(X)\cap \mathscr O_{X\an}(X\an)$ and let $\mathscr I\subset \mathscr O_X$ be
the sheaf of denominators of $f$ viewed as a meromorphic
function on $X$.  By Lemma \ref{deno}
above, $\mathscr I\an$ is the sheaf of denominators of $f$ viewed as a
meromorphic function on $X\an$. The fact that $f\in \mathscr O(X\an)$ means that 
$1\in \mathscr I\an(X\an)$, which we can reformulate by saying that $1$ is zero in $(\mathscr O_{X\an}/
\mathscr I\an)(X\an)$. But by
faithfull flatness of $X\an \to X$, the map
$(\mathscr O_X/\mathscr I)(X)\to (\mathscr O_{X\an}/\mathscr I\an)(X\an)$ is injective. 
Therefore $1$ is zero in $(\mathscr O_X/\mathscr I)(X)$, so
$1\in \mathscr I(X)$ and $f\in \mathscr O_X(X)$. 
\end{proof}

\begin{prop}[See also Theorem 8.7 of \cite{mehmeti2022}]\label{gaga-mero}
Assume that the scheme  $X$ is proper over $A$. 
The natural embedding
$\mathscr K_X(X)\hookrightarrow \mathscr K_{X\an}(X\an)$ is an isomorphism. 
\end{prop}

\begin{proof}
Let $f$ be a meromorphic function on
$X\an$. By GAGA for coherent sheaves (see \cite[Appendix A]{poineau2010}), the sheaf denominators of $f$ is 
of the form $\mathscr I\an$ for a (uniquely determined) sheaf of ideals $\mathscr I$ on 
$X$. Multiplication by $f$ defines a morphism from $\mathscr I\an$ to $\mathscr O_{X\an}$, which once again
by GAGA comes from a unique map 
from $\mathscr I$ to $\mathscr O_X$.
In order to prove that $f$ belongs to $\mathscr K_X(X)$, it suffices to prove that on every
affine open subscheme
$U$ of $X$,
the sheaf $\mathscr I$
has a regular  section $h$ (for if one denotes by $g$ the image of $h$ in in $\mathscr O(U)$ one will have
$f=g/h$ on $U\an$). So let $U$ be an affine open subscheme of $X$. Supppose
that $\mathscr I(U)$ only consists of non-regular elements. This means that $\mathscr I(U)$
is contained in the union of all associated primes of $\mathscr O_X(U)$, which implies that it is contained
in one of them. So there exists a non-zero element $b$ of $\mathscr O_X(U)
\subset \mathscr O_{X\an}(U\an)$ such that $hb=0$ for all
$h\in \mathscr I(U)$. By the definition of $\mathscr I\an$ there is an affinoid G-covering $(V_i)$ of $U\an$ and for every $i$ a section $h_i$ of $\mathscr I\an(V_i)$ which is a regular
element of $\mathscr O_X(V_i)$. Since $h_i$ is a section
of $\mathscr I\an$ we have $h_ib|_{V_i}=0$. As $h_i$ is regular, this implies that $b|_{V_i}=0$. 
Since the $V_i$'s G-cover $U\an$ this yields  $b=0$, contradiction.
\end{proof}

For proving Theorem \ref{main}
we will need three non-archimedean analogues of classical complex-analytic
results, all due to Bartenwerfer in the strict case. We will prove them in general 
by reducing to the strict case. For doing this, we shall use the following notation. If
$r$ is a $k$-free polyradius, that is, a finite family $(r_1,\ldots, r_n)$ of positive
numbers multiplicatively independent modulo $\abs{k^\times}$, we denote by 
$k_r$ the set of all power series $\sum_{I\in \Z^n}a_IT^I$ where the $a_I$'s belong to 
$k$ and $\abs{a_I}r^I\to 0$ when $\abs I\to \infty$. The formula 
$\sum a_I T^I\mapsto \max_I \abs{a_I}r^I$ is a multiplicative norm that makes $k_r$
a complete extension of $k$ (see for instance \cite[1.2]{ducros2007}). If $A$
is a $k$-affinoid algebra we set $A_r=A\widehat \otimes_k k_r$. This is the set of 
power-series $\sum_{I\in \Z^n}a_IT^I$  where the $a_I$'s belong to 
$k$ and $\|a_I\|_\infty r^I\to 0$ when $\abs I\to \infty$. 

\begin{prop}[Bartenwerfer]\label{bartenwerfer}
Let $X$ be a $k$-analytic space and let $Z$ be 
a Zariski-closed subset of $X$ with empty interior; set $U=X\setminus Z$. 
\begin{enumerate}[1]

\item If $X$ is normal any bounded holomorphic function on $U$ extends to a holomorphic function on 
$X$. 
\item If $X$ is reduced and $Z$ is everywhere of codimension $\geq 2$ in $X$, 
then every meromorphic
function on $U$ extends to a meromorphic function on $X$. 
\item If $X$ is normal and $Z$ is everywhere of codimension $\geq 2$ in $X$, then 
every holomorphic
function on $U$ extends to a holomorphic  function on $X$. 
\end{enumerate}
\end{prop}

\begin{proof}
If $k$ is non-trivially valued and $X$ is strictly $k$-analytic, 
(1) is \cite[\S 3, Theorem]{bartenwerfer1976}
and (2) and (3) are particular cases of
\cite[Theorem]{bartenwerfer1975}. We are just simply
going to explain how to extend these statements 
when
$X$ is not assumed to be strict. Everything being G-local on $X$,
we can assume that $X$ is affinoid; and since for proving (2) one can replace $X$
with its normalization, we can assume that $X$ is normal and,  by arguing componentwise, 
irreducible. Let $A$ be the algebra of analytic functions on $X$. 
Let $r$ be a $k$-free polyradius such that
$\abs{k_r}^\times\neq \{0\}$ and $A_r$ is $k_r$-strict. 
The ring $A_r$ is a normal integral 
domain -- this follows for instance from
\cite[Exemple 3.3, \Th 3.1 and 3.3]{ducros2009}, but 
an elementary proof, 
can be found in \cite[Appendix]{ducros2003}.
Then Bartenwerfer's statements apply on the $k_r$-analytic space $X_r$, 
and for deducing them on $X$ it suffices to prove the following: 

\begin{enumerate}[a]
\item if $f=\sum a_I T^I$ is an element of $A_r$ whose
restriction to $U_r$ belongs to $\mathscr O_X(U)$ then $f\in A$. 
\item If $g=\sum b_I T^I$ and $h=\sum c_IT^I$ are two elements
of $A_r$ with $h\neq 0$ such that the restriction of $g/h$ to $U_r$
belongs to $\mathscr K_X(U)$ (where $\mathscr K_X$ is the sheaf of
meromorphic functions on $X$), then $g/h$ belongs to the ring of fractions
of $A$. 
\end{enumerate}

Let us prove (a). If $V$ is an affinoid domain of $U$ then by assumption
the element $\sum (a_I|_V)T^I$ of $\mathscr O_X(V)_r$ belongs to $\mathscr O_X(V)$, 
which means that $a_I|_V=0$ as soon as $I\neq 0$. Since this holds for arbitrary $V$
we see that $a_I|_U=0$ as soon as $I\neq 0$. As $X$ is reduced and $U$ is dense
one has $a_I=0$ as soon as $I\neq 0$, which proves that $f\in A$. 

Let us prove (b). Let 
$E$
denote the set of indices $I$ with $c_I\neq 0$ (it is non-empty by assumption). 
Let $V$ be a connected and non-empty affinoid domain of $U$.
As $X$ is normal $\mathscr O_X(V)$
is an integral domain, and by assumption there exist $u$ and 
$v$ in $\mathscr O(V)$ with $v\neq 0$ and $g/h=u/v$ as meromorphic functions
on $V_r$, which amounts to the equality $gv=hu$ in $\mathscr O_X(V)_r
=\mathscr O_{X_r}(V_r)$. 
As a consequence, if $I\in E$ then $u/v=b_I/c_I$ as meromorphic
functions on $V$ (note that $c_I|_V\neq 0$ since
$V$ is non-empty and $X$ is irreducible
and reduced) and $b_I|_V=0$ if $I\in E$. 
Therefore all meromorphic functions
$b_I/c_I$ for $I\in E$ coincide on every irreducible affinoid domain $V$ of the normal space $U$, 
then coincide on the whole of $U$, and then of $X$ since the latter is reduced ; let $w$ denote the 
common value of the $b_I/c_I$ for $I\in E$, as
a meromorphic function on $X$. 
If $I\notin E$ then $b_I=0$ on every irreducible affinoid domain $V$ of the normal space $U$, 
then on the whole of $U$, and then on $X$ since the latter is reduced. Hence $g/h=w$, which ends
the proof. 
\end{proof}

\section{The proofs}\label{proofs}

\subsection{Proof of Lemma \ref{lemma-disc}}\label{proof-disc}
Implication (i)$\Rightarrow$(ii)
is obvious, so let us assume (ii)
and prove (i). Since $f$ is invertible on a punctured neighborhood of the origin, 
and since the assertion only concerns the singularity of $f$ at the origin, 
we can assume up to shrinking $D$ that $f$ is an invertible function on $U$. 
Let $t$ denote the coordinate function on $D$ and let $R$ be the radius 
of $D$. The function 
$f$ can be written as a power series $f=\sum_{i\in \Z} a_i t^i$ and since $f$
is invertible, 
Lemma \ref{dominant-monomial} ensures that
there exists some $j$ such that $\abs{a_j}r^j>\abs{a_i}r^i$ for
every $i\neq j$ and every $r\in (0,R)$ (and even $(0,R]$ if $D$ is closed). 
By letting $r$ tend to zero we see that $a_i=0$ for every $i<j$. Then 
$f=\sum_{i\geq j}a_i t^i$ admits a meromorphic extension to $d$. \qed

\subsection{Proof of Theorem \ref{theo-discwise}}\label{proof-quasi-main}
Let $\mathscr K_X$ be the sheaf of meromorphic functions on the
space $X$ 
and let $(X_i)_{i\in I}$ be the family of its reduced irreducible components. Let $X'$ be the normalization of
$X$. We have $\mathscr K_X(X)=\mathscr K_{X'}(X')$: indeed, this can be checked G-locally 
and thus enables to assume that $X$ is affinoid, in which case this just comes from the corresponding
scheme-theoretic statement. As a consequence, $\mathscr K_X(X)=\prod_i \mathscr K_{X_i}(X_i)$. 
Let $J$ be the set of indices $i$ such that 
$X_i\cap U\neq
\emptyset$ and set
$V=\bigcup_{i\in J}X_i$, 
equipped with its reduced structure.
Now, every morphism $\phi$ as in (ii)
factors through $V$, and if $f|_{V\cap U}$ admits a meromorphic extension
$(g_i)_{i\in  J}$ to $V$, then $f$ admits
a meromorphic extension to the whole of $X$,  namely $(f_i)_{i\in I}$ with
$f_i=g_i$ if $i\in J$ and $f_i=0$ otherwise. Hence for proving our theorem we can 
replace $X$ by $V$ and thus assume that $U$ is dense. 

\subsubsection{}
Assume that (i) holds and let $\phi$ as in (ii). Let $W$ be an affinoid domain of 
$X$ containing $\phi(x)$. The pre-image $\phi^{-1}(W)$ is an analytic domain of $D$ containing 
the origin, so this is a neighborhood of the later in $D$. Up to shrinking $D$ (this is harmless
for proving (ii), which is a local property at the origin) we can thus assume that $\phi(D)\subset W$; 
otherwise said we can assume that $X$ is affinoid, say $X=\mathscr M(A)$. 
By assumption one can write $f=g/h$ with $h$ a non-zero divisor of $A$. 
As $U$ is dense in $X$, the restriction $h|_U$ is not identically zero, so that
$\phi^*h|_{D\setminus \{0\}}$ is not identically $0$. This ensures that $\phi^*h$ is a regular
holomorphic function on the reduced irreducible space $D$, and in view of the equality 
$\phi^*f\phi^*h=\phi^*g$ on $D\setminus \{0\}$, the function $\phi^*f$ admits
the meromorphic extension $\phi^*g/\phi^*h$ to the whole of $D$, whence (ii). 

\subsubsection{}
We now start the proof of implication (ii)$\Rightarrow$(i), 
in several steps. So we assume that (ii) holds. 
We shall use implicitly several times the following fact: 
since the triple $(X,U,f)$ satisfies (ii), for every analtyic space $Y$ defined over
a complete extension $K$ of $k$ and every $k$-morphism $\psi \colon Y\to X$
the triple $(Y,\psi^{-1}(U),\psi^*f)$ still satisfies (ii) (over the ground field
$K$). 

As $U$ is a dense Zariski-open subset of the reduced space $X$, 
the function $f$ as at most one meromorphic extension to $X$. 
This uniqueness property also holds on any analytic domain of $X$, which will enable us to prove (i)
by argueing G-locally: 
indeed, since G-local meromorphic extensions will be canonical, we will be able to glue them. 

We make a first use
of our right to argue G-locally by assuming
that $X$
is affinoid, say $X=\mathscr M(A)$. 
We denote by $p$ the characteristic exponent of $k$. 

\subsubsection{}\label{radicial-basechange}
Let $n$ be an integer and set $L=k^{1/p^n}$. 
Assume that the image of $f$ in $\mathscr O(U_{L,\mathrm{red}})$ (which 
we shall still denote by $f$) has a meromorphic extension
to $X_{L,\mathrm{red}}$. 
We are going to prove that $f$
has a meromorphic extension to $X$. Note that
for every affinoid $k$-algebra $B$ one has 
$(B\widehat\otimes_k L)^{p^n}\subset B$
(check it on $B\otimes_k L$ and use a limit
argument), which implies by arguing G-locally
that $\mathscr O_X(U_L)^{p^n}\subset \mathscr O_X(U)$. 

By our assumption 
there exist $g$ and $h$ in $A_L:=A\widehat\otimes_k L$
with $h$ regular (\ie, not a zero divisor)
in $A_{L,\mathrm{red}}$ such that $f=g/h$ in $\mathscr O(U_{L,\mathrm{red}})$. We have then $fh=g+N$ in 
$\mathscr O(U_L)$ for some nilpotent function $N$ on $U_L$ (the sheaf of locally nilpotent functions on $U_L$ is the restriction of
the coherent sheaf on $X_L$ associated with the nilradical of the noetherian ring $A$, so its sections are actually nilpotent). It follows that
for $m\geq n$ large enough $f^{p^m}h^{p^m}=g^{p^m}$ in $\mathscr O(U_L)$; but since
$m\geq n$ all functions
involved in this equality belong to the subring $\mathscr O(U)$ of $\mathscr O(U_L)$, so $f^{p^m}h^{p^m}=g^{p^m}$  should be understood as
an equality between analytic functions on $U$.

Set $H=h^{p^m}$ and $G=g^{p^m}$, so that we have $f^{p^m}H=G$
on $U$. Both $G$ and $H$ belong to $A$. Let us show that 
$H$ is a regular element of $A$. 
As $X$ is reduced, it suffices to show that the zero-locus of $H$ does not contain any irreducible component of $X$. But since $h$
is a regular element of $A_{L,\mathrm{red}}$, the zero-locus of $h$ in $X_L$, which is the same as the 
zero-locus of $H$, does not contains any irreducible component of $X_L$,
and we conclude by noticing that $X_L\to X$ induces 
a homeomorphism for the Zariski topologies on the source and the target.

Since $f^{p^m}H=G$ one has $(fH)^{p^m}=GH^{p^m-1}$ on $U$; therefore the analytic function $fH$ is bounded on $U$. By the non-archimedean version of Riemann's extension theorem 
(essentially due to Bartenwerfer, see Proposition  \ref{bartenwerfer} (1)
for more details), 
$fH$ extends to a holomorphic function on the normalization $X'$ on $X$. Therefore $fH$ extends to a
meromorphic function on $X$ and so does $f$ since $H$ is regular.

\subsubsection{}
If $n$ is large enough, a result by Conrad \cite[Lemma 3.3.1]{conrad-1999}, see also \cite[\Th 6.10]{ducros2009},
ensures that the normalization of $X_{k^{1/p^n}}$ is geometrically normal. 
In view of \ref{radicial-basechange}, it suffices to 
check that the image of $f$ in $\mathscr O(U_{k^{1/p^n},\mathrm{red}})$
extends to a
meromorphic function on $X_{k^{1/p^n},\mathrm{red}}$. As a consequence, 
we may assume that the normalization $X'$ of $X$ is geometrically normal
(and then $X'_L$ is the normalization of $X_L$ for every complete extension $L$ of $k$, see
\cite[\Prop 5.20]{ducros2009}). Let $U'$ be the pre-image of $U$ in $X'$.

Another result
by Conrad \cite[\Th 3.3.8]{conrad-1999}, 
see also \cite[\Th 6.11]{ducros2009}, ensures that for large enough $n$, the space
$(X'_{k^{1/p^n}}\setminus U'_{k^{1/p^n}})_{\mathrm{red}}$ is geometrically reduced. So using again \ref{radicial-basechange}, we may moreover assume that 
$(X'\setminus U')_{\mathrm{red}}$ is geometrically reduced.

\subsubsection{} 
By replacing $X$ with $X'$ and $U$ with $U'$
we can assume that $X$ is normal. And then by arguing componentwise we can 
assume that $X$ is moreover integral. Let $n$ denote its dimension. As $U$ is dense, $X\setminus U$ is a Zariski-closed subset of $X$
of dimension $\leq n-1$. Let $K$ be the completion of an algebraic closure of $k$. 
Let $Y$ denote the complement of the quasi-smooth locus of $X$. The space 
$X_K$ is normal, therefore $\dim Y_K\leq n-2$, hence $\dim Y\leq n-2$. 
And the space $(X\setminus U)_{\mathrm{red},K}$ is reduced, 
so it has a dense quasi-smooth locus. Therefore $(X\setminus U)_{\mathrm{red}}$
has a dense quasi-smooth locus. 
Let us denote by $Y'$
be the union of $Y$, of all irreducible
components of $X\setminus U$ of
dimension
$\leq n-2$, and of the non-quasi-smooth locus
of $(X\setminus U)_{\mathrm{red}}$. 
By construction, $Y'$ is a Zariski-closed subset of $X$ of dimension 
$\leq n-2$. As $X$ is  reduced, it follows from an extension theorem
essentially due to Bartenwerfer
(see Proposition \ref{bartenwerfer} (2) for more details)
that every meromorphic function on  $X\setminus Y'$ extends to 
a meromorphic function on $X$. 
It is therefore sufficient to prove that $f|_{V\cap U}$ extends to a
meromorphic function on $V$ for every
affinoid domain $V$ of $X\setminus Y'$. Hence we
have reduced to the case where $X$ is quasi-smooth and $U$ is the complement
of a quasi-smooth hypersurface $S$.

\subsubsection{}
For proving the theorem we may once again argue locally on
$X$. The theorem obviously holds on $X\setminus S$ , 
and if $x$ is a point of $S$ it follows from 
Lemma \ref{structure-smooth}
that there exists an affinoid neighborhood $V$ of $x$ in $X$
such that $(V,S\cap V)$ is isomorphic
to $D\times_k (S\cap V)\times D,\{0\}\times_k  (S\cap V))$ for some 
closed one-dimensional disc $D$. Moreover
we can shrink $V$
so that
the smooth space $S
\cap V$ is connected, hence irreducible (and reduced). 

Therefore we can assume that $X=D\times_k Y$ for some
irreducible
and reduced analytic space
$Y$ and some closed
disc $D$
and that $S=\{0\}\times_k Y$.
(The quasi-smoothness of $S\simeq 
Y$ was useful
for reducing to this
product situation, but will not be used anymore; the fact that $Y$
is irreducible and reduced
will  be sufficient.)

Let $t$ be the coordinate function on $D$. 
By hypothesis, $f$ is an analytic
function defined on $(D\setminus \{0\})\times_k Y$, 
so it can be written $\sum_{i\in \mathbf 
Z}b_i t^i$ where every $b_i$ belongs to 
$\mathscr O(Y)$ (\ref{functions-annulus}). 
Choose $y\in Y$ lying over the generic point
of $\mathrm{Spec}
\,\mathscr O(Y)$.
The fiber of $X=D\times_k  Y$ over $y$ (through the second
projection) is canonically isomorphic to 
$D_{\hr y}$; let
$\phi\colon D_{\hr y}\to X$ be the corresponding 
embedding. By construction, $\phi^{-1}(U)
=D_{\hr y}\setminus \{0\}$. It thus follows from our 
assumption (ii) that $\phi^*f$ admits a meromorphic
extension to $D_{\hr y}$. Since $\phi^*f=\sum_i b_i(y)t^i$, 
this means that there is some $j$ such that
$b_i(y)=0$ for every $i\leq j$.
As $y$
is Zariski-generic on the reduced, irreducible space $Y$,
this implies that $b_i=0$ for every $i<j$. 
Thus $f=\sum_{i\geq j}b_i t^i$ extends to a meromorphic function on
$X$, which ends the proof of Theorem \ref{theo-discwise}.\qed

\subsection{Proof of Theorem \ref{main}}\label{proof-main}
Let us assume that (i) holds. Let us choose a meromorphic extension of $f$, which 
we still denote by $f$, and let $\mathscr I$ be its sheaf of denominators. The 
$\mathscr J:=\mathscr If$ is a coherent sheaf of ideals on $X$, whose restriction 
to $U$ is equal to $(f|_U)$ since $\mathscr I|_U=\mathscr O_U$. 
Therefore the zero-locus of $f|_U$ is the intersection of $U$ and of the zero-locus of
$\mathscr J$, which is a Zariski-closed subset of $X$, whence (ii). 

Conversely, assume that (ii) holds. Let $L$ be a complete extension of 
$k$, let $D$ be a one-dimensional 
disc over $L$ and let $\phi\colon D\to X$ be a $k$-morphism such that $\phi^{-1}(U)=D\setminus \{0\}$. The zero-locus of
the function $\phi^*f\in \mathscr O_D(D\setminus \{0\})$ is then equal to $\phi^{-1}(Z)\cap (D\setminus \{0\})$. Therefore by Lemma
\ref{lemma-disc}
which we have already proved
(\ref{proof-disc}), 
$\phi^*f$ admits an extension to a meromorphic function of $D$. Since this holds for arbitrary $(L,D,\phi)$, Theorem
\ref{theo-discwise} proven 
in \ref{proof-quasi-main}
ensures that $f$ admits an extension to a meromorphic function on $X$, 
which ends the proof of Theorem \ref{main}. \qed

\subsection{Proof of Corollary \ref{cor}}
The statement is local on $X$, which 
enables us to assume that $X$
is affine. Let us choose a 
reduced projective compactification $\overline X$
of $X$. The zero-locus of $f$ on $X\an$
is then equal to $\overline Y\an\cap X\an$
(where $\overline Y$ is the closure
of $Y$ in
$\overline X$) so that we can apply
Theorem \ref{main} proved in \ref{proof-main}
(taking $X=\overline X\an$ and $U=X\an$)
and conclude that the function $f$ on $X\an$
``is" then a meromorphic function on $\overline X\an$. By GAGA for meromorphic functions
(see \Prop \ref{gaga-mero}), $f$ ``is" a meromorphic function 
on $\overline X$. Then $f$ is a meromorphic function on the scheme $X$
inducing a holomorphic function on $X\an$. By Lemma \ref{mero-holo}, 
$f\in \mathscr O_X(X)$. 
\qed

\section*{Acknowledgements}
The starting point of this work was a discussion with Marco Maculan who drew my attention to the fact that
invertible analytic functions on analytifications of curves are algebraic. This is what led me to Theorem \ref{main}
and Corollary~\ref{cor}, and I would like to thank him warmly for that. 
I am also very grateful to Piotr Achinger for his careful reading a former version of this manuscript, his
insightful comments and his very interesting suggestion described in Remark~\ref{question-piotr} at the beginning of this text. Finally I am thankful to the two anonymous referees whose remarks,
suggestions and comments helped on several versions of this work me to signficantly improve the manuscript.

\bibliographystyle{smfalpha}
\bibliography{aducros}

\end{document}